\documentclass[reqno]{amsart}
\usepackage[english]{babel}
\usepackage{amscd,amssymb,amsmath,amsfonts,latexsym,amsthm}
\usepackage{inputenc}
\usepackage{graphicx}
\newtheorem{thm}{Theorem}[section]
\newtheorem{cor}[thm]{Corollary}

\newtheorem{prop}[thm]{Proposition}
\newtheorem{defn}[thm]{Definition}

\newtheorem{rem}[thm]{Remark}
\numberwithin{equation}{section}
\DeclareMathOperator{\Der}{Der}
\DeclareMathOperator{\Hom}{Hom}
\DeclareMathOperator{\Orb}{Orb}
\begin{document}


\title[Infinitesimal deformations of null-filiform Leibniz superalgebras]
{Infinitesimal deformations of null-filiform Leibniz superalgebras}%
\author{Khudoyberdiyev A.Kh. and Omirov B.A.}

\address{[A.\ Kh.\ Khudoyberdiyev and B.\ A.\ Omirov] Institute of Mathematics, National University of Uzbekistan,
Tashkent, 100125, Uzbekistan.} \email{khabror@mail.ru, omirovb@mail.ru}\

%

\begin{abstract}
In this paper we describe the infinitesimal deformations of
null-filiform Leibniz superalgebras over a field of zero
characteristic. It is known that up to isomorphism in each
dimension there exist two such superalgebras $NF^{n,m}$. One of
them is a Leibniz algebra (that is $m=0$) and the second one is a pure
Leibniz superalgebra (that is $m\neq 0$) of maximum nilindex. We
show that the closure of union of orbits of single-generated Leibniz
algebras forms an irreducible component of the variety of Leibniz
algebras. We prove that any single-generated Leibniz algebra is
a linear integrable deformation of the algebra $NF^{n}$. Similar
results for the case of Leibniz superalgebras are obtained.

\end{abstract}

\maketitle \textbf{Mathematics Subject Classification 2010}:
17A32, 17A70, 17B30, 13D10, 16S80.


\section{Introduction.}

Deforming a given mathematical structure is a tool of fundamental
importance in most parts of mathematics, mathematical physics and
physics. Deformations and contractions have been investigated by
researchers who had different approaches and goals. Tools such as
cohomology, gradings, etc. which are utilized in the study of one
concept, are likely to be useful for the other concept as well.

The theory of deformations originated with the problem of
classifying all possible pairwise non-isomorphic complex
structures on a given differentiable real manifold. Formal
deformations of arbitrary rings and associative algebras, and
related cohomology questions, were first investigated by
Gerstenhaber \cite{Gersten}. Later, the notion of deformation was
applied to Lie algebras by Nijenhuis and Richardson \cite{Nijen}.
Because various fields in mathematics and physics exist in which
deformations are used, we focus in the study of Leibniz superalgebras. One-parameter deformations were studied and
established connection between Lie algebra cohomology and
infinitesimal deformations.

Deformation is one of the tools used to study a specific object,
by deforming it into some families of ``similar'' structure objects.
This way we get a richer picture about the original object itself
\cite{Fial1}. But there is also another question approached via
deformation. Roughly speaking, it is the question, can we equip
the set of mathematical structures under consideration (may be up
to certain equivalence) with the structure of a topological or
geometric space.

The theory of deformations is one of the effective approach in
investigating of solvable and nilpotent Lie algebras and
superalgebras \cite{Fial2,Fial3,Khak,Mil},
etc.

Recall, that Leibniz algebras are generalization of Lie algebras
\cite{Lo 2,Lo 3} and it is natural to apply the theory of
deformations to the study of Leibniz algebras. Particularly,
the problems which were studied in \cite{Fial2,Mil} and
others can be considered from point of Leibniz algebras view. Thanks
to the work \cite{Balavoine} we can apply the general principles for
deformations and rigidity of Leibniz algebras.

It is well known that Lie superalgebras are a generalization of
Lie algebras. In the same way, the notion of Leibniz algebra, can
be generalized to Leibniz superalgebras. Lie superalgebras with
maximal nilindex were classified in \cite{Khak1}. In fact, there
exists a unique Lie superalgebra of maximal nilindex. This
superalgebra is a filiform Lie superalgebra. For nilpotent Leibniz
superalgebras the description of the maximal nilindex case
(nilpotent Leibniz superalgebras distinguished by the feature of
being single-generated) was easily done in \cite{Alb}.

 Let $V=V_0\oplus V_1$ be the underlying
vector space of the Leibniz superalgebra $L=L_0\oplus L_1$ of
dimension $n+m$ (where $n$ and $m$ are dimensions of $L_0$ and
$L_1$, respectively)  and let $GL(V)$ be the group of the
invertible linear mappings of the form $f = f_0 + f_1$ such that
$f_0\in GL_n(F)$ and $f_1\in GL_m(F)$ (where $GL(V)=GL_n(F)\oplus
GL_m(F)$). The action of the group $GL(V)$ on the variety of
Leibniz superalgebras induces an action on the Leibniz
superalgebras variety: two laws $\mu_1$ and $\mu_2$ are
isomorphic, if there exists a linear mapping  $f$, $f =f_0 +
f_1\in GL(V)$, such that
$$ \mu_2(x,y)=f_{\alpha +\beta}^{-1}(\mu_1(f_{\alpha}(x), f_{\beta}(y)))\ \ \mbox{for \  all}
\ x\in V_{\alpha}, y\in V_{\beta}.$$
 The orbit under this action,
denoted by $\Orb(\mu)$, consists of all superalgebras isomorphic
to the superalgebra $\mu$. Therefore the description of
$(n+m)$-dimensional superalgebras with dimensions of even and odd parts equal to $n$ and $m$,
respectively (further denoted by $Leib^{n,m}$) can be reduced to a geometric problem of
classification of orbits under the action of the group $GL(V)$.
Note that nilpotent Leibniz superalgebras $N^{n,m}$ form also an
invariant subvariety of the variety $Leib^{n,m}$ under above
action. From algebraic geometry it is known that an algebraic
variety is a union of irreducible components. The superalgebras
with open orbits in the variety of Leibniz superalgebras are {\it
called rigid}. The closures of these open orbits give irreducible
components of the variety. Therefore studying of the rigid
superalgebras is a crucial problem from the geometrical point of
view. The problem of finding such algebras is crucial for the
description of the variety $Leib^{n,m}$.

The structure of the paper is as follows: In Section Preliminaries
we give the necessary definitions and results for understanding
the main parts of this paper. In Section 3 we calculate the second
group of cohomology of the null-filiform Leibniz algebra and show
that the set of single-generated Leibniz algebras forms an
irreducible component of the variety of Leibniz algebras.
Moreover, it established that any single-generated algebra is a
linear integrable deformation of the null-filiform algebra. In the
last section we extend the calculations of previous section for
the case of Leibniz superalgebras.

Throughout the paper we consider finite-dimensional vector spaces
and superalgebras over a field of zero characteristic. Moreover,
in the  multiplication table of a Leibniz superalgebra the omitted
products and in the expansion of 2-cocycles the omitted values
are assumed to be zero.

\section{Preliminaries.}

In this section we give necessary definitions and results for
understanding the main parts of the work.

\begin{defn} \cite{Alb} A $\mathbb{Z}_2$-graded vector space $L=L_0\oplus
L_1$ is called a Leibniz superalgebra if it is equipped with a
product $[-, -]$ which satisfies the following conditions:
$$[x, [y, z]]=[[x, y], z] - (-1)^{|y||z|} [[x, z], y]- \mbox{Leibniz  \ superidentity}$$
 for all $x\in L$, $y \in L_{|y|}$, $z \in
L_{|z|}$.
\end{defn}

Let $L$ be a Leibniz superalgebra. We call a $\mathbb{Z}_2$-graded
vector space $M=M_0\oplus M_1$ a module over $L$ if there are two
bilinear maps:
$$[-,-]:L\times M \rightarrow M \qquad \text{and} \qquad [-,-]:M\times L \rightarrow M$$
satisfying the following three axioms
\begin{align*}
[m,[x,y]] & =[[m,x],y]-(-1)^{|x||y|}[[m,y],x],\\
[x,[m,y]] & =[[x,m],y]-(-1)^{|y||m|}[[x,y],m],\\
[x,[y,m]] & =[[x,y],m]-(-1)^{|m||y|}[[x,m],y],
\end{align*}
for any $m\in M_{|m|}$, $x\in L_{|x|}, y\in L_{|y|}$.

Given a Leibniz superalgebra $L$, let $C^n(L,M)$ be the space of
all super skew-symmetric $F$-linear homogeneous mapping
$L^{\otimes n} \rightarrow M$, $n \geq 0$ and $C^0(L,M) = M$. This
space is graded by $C^n(L,M) = C^n_0(L,M) \oplus C^n_1(L,M)$ with
$$C^n_p(L,M) =
\bigoplus\limits_{\begin{array}{c}n_0+n_1=n\\n_1+r\equiv p \mod
2\end{array}}\Hom(L_0^{\otimes n_0}\otimes L_1^{\otimes n_1},
M_r)$$

Let $d^n : C^n(L,M) \rightarrow C^{n+1}(L,M)$ be an $F$-homomorphism
defined by
 \begin{multline*}
(d^nf)(x_1, \dots , x_{n+1}): =
[x_1,f(x_2,\dots,x_{n+1})] \\
+\sum\limits_{i=2}^{n+1}(-1)^{i+|x_i|(|f|+|x_{i+1}|+\dots+|x_{n+1}|)}[f(x_1,
\dots, \widehat{x_i}, \dots , x_{n+1}),x_i]\\
+\sum\limits_{1\leq i<j\leq
{n+1}}(-1)^{{j+1}+|x_j|(|x_{i+1}|+\dots+|x_{j-1}|)}f(x_1, \dots,
x_{i-1},[x_i,x_j], x_{i+1}, \dots , \widehat{x_j}, \dots
,x_{n+1}),
\end{multline*}
 where $f\in C^n(L,M)$ and $x_i\in L$. Since the derived
operator $d=\sum\limits_{i \geq 0}d^i$ satisfies the property
$d\circ d = 0$, the cohomology group is well defined and
$$HL_p^n(L,M) = ZL_p^n(L,M)/ BL_p^n(L,M),$$
where the elements $ZL_0^n(L,M)$ ($BL_0^n(L,M)$) and $ZL_1^n(L,M)$
($ BL_1^n(L,M)$) are called {\it even $n$-cocycles} ({\it even
$n$-coboundaries}) and {\it odd $n$-cocycles} ({\it odd
$n$-coboundaries}), respectively.

It is remarkable fact that formula for $d^n$ can be obtained from
the derived operator for color Leibniz algebras \cite{Dzhum}.

Note that the space $ZL^1(L,L)$ consists of derivations of the
superalgebra $L$, which are defined by the condition:
$$d([x,y])  = (-1)^{|d||y|} [d(x), y] + [x, d(y)].$$

For a given $x \in L$, $R_x$ denotes the map $R_x: L \rightarrow
L$ such that $R_x(y)=[y,x], \ \forall x \in L$. Note that the map
$R_x$ is a derivation.

{\it A deformation of a Leibniz superalgebra} $L$ is a
one-parameter family $L_t$ of Leibniz superalgebras with the
bracket $$\mu_t = \mu_0 + t\varphi_1 + t^2\varphi_2 + \cdots,$$
where $\varphi_i$ are $L-$valued even 2-cochains, i.e., elements of
$\Hom(L\otimes L, L)_0 = C^2(L, L)_0$.

Two deformations $L_t, \ L'_t$ with corresponding laws $\mu_t, \
\mu'_t$ are {\it equivalent} if there exists a linear automorphism
$f_t = id + f_1 t + f_2 t^2 + \cdots$ of $L$, where $f_i$ are
elements of $C^1(L, L)_0$ such that the following equation holds
$$\mu'_t(x, y) = f_t^{-1}(\mu_t(f_t(x), f_t(y))) \ \ \text{for} \  x, y \in L.$$

The Leibniz superidentity for the superalgebras $L_t$ implies that
the 2-cochain $\varphi_1$ is an even 2-cocycle, i.e. $d^2\varphi_1
= 0$. If $\varphi_1$ vanishes identically, the first non vanishing
$\varphi_i$ will be a 2-cocycle.

If $\mu'_t$ is an equivalent deformation with cochains
$\varphi_i'$, then $\varphi_1' -\varphi_1 = d^1f_1$, hence every
equivalence class of deformations defines uniquely an element of
$HL^2(L, L)_0$.

Note that the linear integrable deformation $\varphi$ satisfies the
condition
\begin{equation}\label{E:2.1}
\varphi(x,  \varphi(y, z)) -  \varphi(\varphi(x, y),
z) + (-1)^{|y||z|}\varphi(\varphi(x, z), y) = 0.
\end{equation}

It should be noted that a Leibniz algebra is a superalgebra with
trivial odd part and the definition of cohomology groups of
Leibniz superalgebras extend the definition of cohomology groups
of Leibniz algebras given in \cite{Lo 3}.

For a Leibniz superalgebra $L$ consider the following central lower
series:
$$
L^1=L,\quad L^{k+1}=[L^k,L^1], \quad k \geq 1.
$$

\begin{defn} \label{defn22} A Leibniz superalgebra $L$ is said to be
nilpotent,
if there exists  $p\in\mathbb N$ such that $L^p=0$.
\end{defn}

Now we give the notion of null-filiform Leibniz superalgebra.
\begin{defn} An $n$-dimensional Leibniz superalgebra is said to be null-filiform if
$\dim L^i=n+1-i, \ 1\leq i \leq n+1$.
\end{defn}

Similarly to the case of nilpotent Leibniz algebras \cite{Ayup} it
is easy to check that a Leibniz superalgebra is null-filiform if and
only if it is single-generated. Moreover, a null-filiform
superalgebra has the maximal nilindex.

\begin{thm} \label{t1} \cite{Alb} Let $L$  be a null-filiform Leibniz superalgebra of the variety $Leib^{n,m}$. Then $L$ is isomorphic to one
of the following non-isomorphic superalgebras:
$$ NF^{n}: \ [x_i,x_1]=x_{i+1},\ 1\le i\le n-1; \quad NF^{n,m}: \
 \begin{cases}[y_i, y_1] = x_{i},& 1 \leq i \leq n,\\
[x_i, y_1] = \frac 1 2 y_{i+1},& 1 \leq i \leq m-1,\\
[y_j, x_1] = y_{j+1},& 1 \leq j \leq m-1,\\
[x_i, x_1] = x_{i+1},& 1 \leq i \leq n-1.
\end{cases}$$
where $\{x_1, x_2, \dots, x_n\}$ and $\{y_1, y_2, \dots, y_m\}$ are
bases of the even and odd parts, respectively.
\end{thm}

\begin{rem} {\em Note that the first superalgebra is a null-filiform Leibniz
algebra \cite{Ayup} and from the assertion of Theorem \ref{t1} we
conclude that in the case of non-trivial odd part of the
null-filiform Leibniz superalgebra $NF^{n,m}$ there are two
possibilities for $m$, namely, $m=n$ or $m=n+1$.}
\end{rem}

\section{Deformations of the null-filiform Leibniz algebra}

In this section we calculate infinitesimal deformations of the
algebra $NF^n$ and we show that any single-generated is a linear
integrable deformation of $NF^n$.

Note that any derivation of the null-filiform Leibniz algebra $NF^{n}$
has the following form \cite{Cas1}:
$$
\left(\begin{array}{ccccc}
a_1& a_2&a_3&\ldots&a_n\\
0& 2a_1&a_2&\ldots& a_{n-1}\\
0& 0&3a_1&\ldots& a_{n-2}\\
\vdots&\vdots&\vdots&\ldots&\vdots\\
0&0&0&\ldots&na_1\\
\end{array}\right).
$$

From this we conclude that $\dim BL^2(NF^{n}, NF^{n}) = n^2 - n$.

In general, a 2-cocycle is a bilinear map from $NF^{n}\otimes
NF^{n}$ to $NF^{n}$ such that $d^2\varphi =0$, i.e.,
$$d^2\varphi(x, y, z) =
[x,\varphi(y,z)] - [\varphi(x,y), z] + [\varphi(x,z), y] +
\varphi(x, [y,z]) - \varphi([x,y],z) + \varphi([x,z],y).$$

\begin{prop}\label{prA1} The following cochains:
$$\varphi_{j,k}(x_j, x_1) = x_k, \  1 \leq j \leq n, \ 2 \leq k \leq n,$$
$$\psi_j  \ (1 \leq j \leq n-1) = \left\{\begin{array}{ll}\psi_j(x_j, x_1) =
x_1,& \\ \psi_j(x_i, x_{j+1}) =-x_{i+1},& 1 \leq i \leq
n-1,\end{array}\right.
$$ form a basis of $ZL^2(NF^{n},NF^{n})$.
\end{prop}

\begin{proof} Using the Leibniz 2-cocycle property
 $(d^2\varphi )(x_i, x_1, x_1) = 0$, we have
\begin{equation}\label{E:3.1}
\varphi(x_i, x_2) = - [x_i,  \varphi(x_1, x_1)], \quad 1 \leq i \leq n-1, \quad \varphi(x_n, x_2) = 0.
\end{equation}

The conditions $(d^2 \varphi)(x_i, x_1, x_j) = 0$, $(d^2
\varphi)(x_i, x_j, x_1) = 0$ for $1 \leq i \leq n$, $ 2 \leq j
\leq n$ imply $$[x_i,\varphi(x_1,x_j)] + [\varphi(x_i,x_j), x_1]
-  \varphi([x_i,x_1],x_j) =0,$$
$$[x_i,\varphi(x_j,x_1)] - [ \varphi(x_i,x_j), x_1]  +
\varphi(x_i, [x_j,x_1])  +  \varphi([x_i,x_1],x_j)=0.$$

Summarizing the above equalities, we derive
\begin{equation}\label{E:3.2}
\begin{cases}\varphi(x_i, x_{j+1}) = - [x_i,  \varphi(x_1, x_j)+\varphi(x_j, x_1)], & 1 \leq i \leq n-1, \ 2 \leq j \leq n-1,\\
\varphi(x_n, x_{j+1})=0, & 2 \leq j \leq n-1, \\
[x_i,  \varphi(x_1, x_n)+\varphi(x_n, x_1)] = 0, & 1 \leq i \leq n.
\end{cases}
\end{equation}

Set  $\varphi(x_j, x_1) = \sum\limits_{k=1}^na_{j,k}x_k$ for $1
\leq i \leq n$.

Using inductively method from equalities \eqref{E:3.1} and \eqref{E:3.2} we get
$a_{n,1} =0$ and
$$\varphi(x_i, x_{j+1}) = -a_{j,1}x_{i+1},\quad 1 \leq i \leq n-1,\ 1 \leq j \leq n-1.$$

Therefore, we obtain that any infinitesimal deformation of
$NF^{n}$ has the following form:
$$\begin{cases}\varphi(x_j, x_1) = a_{j,1}x_1 + a_{j,2}x_2+ \dots + a_{j,n}x_n,& 1 \leq j \leq n-1\\
\varphi(x_n, x_1) = a_{n,2}x_2+ \dots + a_{n,n}x_n,& \\
\varphi(x_i, x_{j+1}) = -a_{j,1}x_{i+1}, & 1 \leq i \leq n-1,\ 1 \leq j \leq n-1.
\end{cases}$$

Therefore, $\varphi_{j,k}$ and $\psi_j$ form a basis of
$ZL^2(NF^{n},NF^{n})$.
\end{proof}

\begin{cor}
$\dim(ZL^2(NF^{n},NF^{n})) = n^2-1$.
\end{cor}

Below, we describe a basis of the subspace $BL^2(NF^{n}, NF^{n})$ in
terms of $\varphi_{j,k}$ and $\psi_j$.

\begin{prop}\label{pr4} The cocycles
$$\xi_{j,k}: \quad \begin{cases}\xi_{j,1} = \psi_{j-1} - \varphi_{j,2}, & 2 \leq j \leq n, \\
\xi_{j,k} = \varphi_{j-1,k}, & 2 \leq j \leq k \leq n, \\
\xi_{j,k} = \varphi_{j-1,k}- \varphi_{j,k+1}, & 2 \leq k < j \leq
n\end{cases}$$ form a basis of $BL^2(NF^{n}, NF^{n})$.
\end{prop}
\begin{proof} Consider the endomorphisms $f_{j,k}$ defined as follows:
$$f_{j,k}(x_j) = x_k, \ 2 \leq j \leq n, \ 1 \leq k \leq n.$$

It is easy to see that $f_{j,k}$ are complement of derivations to
$C^1(NF^{n}, NF^{n})$. Therefore, the elements of the space
$BL^2(NF^{n}, NF^{n})$ are $d^1f_{j,k}$ such that $d^1f_{j,k} =
f_{j,k}([x,y]) - [f_{j,k}(x),y] - [x, f_{j,k}(y)]$.

Then we obtain
$$\begin{array}{l} d^1f_{j,1} \ (2 \leq j \leq n)=\begin{cases}d^1f_{j,1}(x_{j-1}, x_1) = x_1, &\\ d^1f_{j,1} (x_j, x_1) =
-x_2,&\\ d^1f_{j,1} (x_i, x_j) = -x_{i+1}, & 2\leq  i \leq n-1,
\end{cases}\\[1mm] d^1f_{j,k} \ \left(\begin{array}{l} 2 \leq j \leq n, \\
2 \leq k \leq n-1\end{array}\right)=\begin{cases}d^1f_{j,k} (x_{j-1}, x_1) = x_k, \\
d^1f_{j,k} (x_j, x_1) = -x_{k+1},\end{cases}\\[1mm]
d^1f_{k,n} \ (2
\leq k \leq n) = \{d^1f_{k,n} (x_{k-1}, x_1) = x_n. \end{array}$$

It should be noted that
$$\begin{cases}d^1f_{j,1} =  \psi_{j-1} - \varphi_{j,2} & 2 \leq j \leq n,\\
d^1f_{j,k} =  \varphi_{j-1,k} -  \varphi_{j,k+1}, & 2 \leq j \leq n, \ 2 \leq k \leq n-1,\\
d^1f_{j,n} =  \varphi_{j-1,n},  &      2 \leq j \leq n.
\end{cases}$$

From the condition $d^1f_{k,s} + d^1f_{k+1,s+1} + \dots +
d^1f_{n+k-s,n} = \varphi_{k-1,s}$ for $2 \leq k \leq s \leq n$, we
conclude that the maps $\xi_{k,s}, \ 2 \leq k \leq n, \ 1 \leq s
\leq n$, form a basis of $BL^2(NF^{n}, NF^{n})$.
\end{proof}

\begin{cor} \label{cor34} The adjoint classes $\overline{\varphi_{n,k}}$ ($2 \leq k \leq n$) form a basis
of $HL^2(NF^{n}, NF^{n})$. Consequently, $\dim HL^2(NF^{n}, NF^{n})
= n - 1$.
\end{cor}

In the following proposition we describe infinitesimal
deformations of $NF^n$ satisfying the equality \eqref{E:2.1}.
\begin{prop}\label{pr35} A 2-cocycle of $NF^{n}$ satisfy the equality \eqref{E:2.1} if and only if it has the
form:
$$\sum\limits_{j,k}a_{j,k}\varphi_{j,k}.$$
\end{prop}
\begin{proof} It is easy to check that 2-cocycles of the form $\sum\limits_{j,k}a_{j,k}\varphi_{j,k}$
satisfy the equality \eqref{E:2.1}.

If $\varphi\in ZL^2(NF^{n}, NF^{n})$, then $\varphi =
\sum\limits_{j,k}a_{j,k}\varphi_{k,s}+
\sum\limits_{j=1}^{n-1}b_{j}\psi_{k}$.

From the condition $$\varphi(x_1,  \varphi(x_1, x_1)) -
\varphi(\varphi(x_1, x_1), x_1) +  \varphi(\varphi(x_1, x_1), x_1)
= 0,$$ we get $b_1=0$.

The following chain of equalities
\begin{align*}
{}& \varphi(x_i,  \varphi(x_j, x_{j+1})) -  \varphi(\varphi(x_i, x_j), x_{j+1}) +  \varphi(\varphi(x_i, x_{j+1}), x_j) \\ {}& =
\varphi(x_i,  \psi_{j}(x_j, x_{j+1})) -  \varphi(\psi_{j-1}(x_i, x_j), x_{j+1}) +  \varphi(\psi_{j}(x_i, x_{j+1}), x_j) \\{}&  =
-\psi_{j}(x_i, b_jx_{j+1}) +  \psi_{j}(b_{j-1}x_{i+1}, x_{j+1}) -
\psi_{j-1}(b_jx_{j+1}, x_j)\\{}&  = b^2_jx_{i+1} -
b_jb_{j-1}x_{i+2}+b_jb_{j-1}x_{i+2} = b^2_jx_{i+1}
\end{align*} imply $b_j=0,
\ 2 \leq j \leq n-1$.
\end{proof}

Consider the linear integrable deformations $\mu_t= NF^{n} +
t\sum\limits_{j,k}a_{j,k}\varphi_{j,k}$ of $NF^{n}$.

Since every non-trivial equivalence class of deformations defines
uniquely an element of $HL^2(L, L)$, due to Corollary \ref{cor34}
it is sufficient to consider $\mu_t(a_2, a_3, \dots, a_n) = NF^{n}
+ t\sum\limits_{k=2}^na_{k}\varphi_{n,k}$, where $(a_2, a_3,
\dots, a_n) \neq (0, 0, \dots, 0)$.

Thus, the  multiplication table of $\mu_t(a_2, a_3, \dots, a_n)$
has the form
$$\begin{cases}[x_i, x_1] = x_{i+1}, & 1\leq i
\leq n-1,\\
[x_n, x_1] = t\sum\limits_{k=2}^na_{k}x_k.\end{cases}$$

Putting $a_k' = ta_k$, we can assume $t=1$.

\begin{prop} An arbitrary single-generated Leibniz algebra
admits a basis $\{x_1, x_2, \dots, x_n\}$ such that the
multiplication table  has the form of $\mu_1(a_2, a_3, \dots, a_n)$.
\end{prop}

\begin{proof}
Let $L$ be a single-generated Leibniz algebra and  $x$  a
generator of $L$. We put $$x_1 = x,\quad x_2 = [x,x], \quad x_3 =
[[x,x],x], \quad \dots, \quad x_n = [[x,x],\dots,x].$$

Since $x$ is a generator, $\{x_1, x_2, \dots, x_n\}$ form a basis
of $L$. Evidently $\{x_2, \dots, x_n\}$ belong to the right
annihilator of $L$. Hence, we have $[x_i, x_j] = 0, \ 2\leq j \leq
n-1$. Let $[x_n, x_1] = \sum\limits_{k=1}^na_{k}x_k$.

From the Leibniz identity $[x_1, [x_n,x_1]] = [[x_1, x_n],x_1] -
[[x_1, x_n],x_1] =0$, we conclude $a_1=0$. Therefore, we obtain
the existence of a basis $\{x_1, x_2, \dots, x_n\}$ in any
single-generated Leibniz algebra such that the
multiplication table in this basis has the form:

\[\left\lbrace \begin{aligned}
{}[x_i, x_1] & = x_{i+1}, & 1\leq i
\leq n-1,\\
[x_n, x_1] &= \sum\limits_{k=2}^na_{k}x_k\,.
\end{aligned}\right.\]
\end{proof}

Let $a_j$ be the first non vanishing parameter in the algebra
$\mu(a_2, a_3, \dots, a_n)$, then by scaling $x_i' = \frac 1
{\sqrt[\uproot{4} n-j+1]{a_j^i}}x_i, \ 1\leq i \leq n$, we can assume
$a_j=1$, i.e., the first non vanishing parameter can be taken equal
to 1.

Note that the set of single-generated Leibniz algebras is open.
Indeed, if a $q$-generated ($q>1$) Leibniz algebra has a basis
$\{e_1, e_2, \dots, e_n \}$, then for any $e_i \in L$ the elements
$e_i, e_i^2, \dots, e_i^n$ are linearly dependent. That is,
determinants of the matrices $A_i, \ 1\leq i \leq n$, which
consists of the rows $e_i, e_i^2, \dots, e_i^n$ are zero, hence we
get $n$-times of polynomials with structure constants of the
algebra. Therefore, $q$-generated ($q>1$) Leibniz algebras form a
closed set. Taking into account that the set of all
single-generated Leibniz algebras is complemented set to a
closed set, we conclude that the set of single-generated Leibniz
algebras is open.

It is easy to see that an algebra $\mu_1(a_2, a_3, \dots, a_n)$ is
a linear deformation of an algebra $\mu_1(a'_2, a'_3, \dots,
a'_n)$.

Since $\dim(\Der(\mu_1(a_2, a_3, \dots, a_n)))=n-1, \ (a_2, a_3,
\dots, a_n)\neq (0, 0, \dots, 0)$, then by arguments used in
\cite{Burde1} for non-isomorphic algebras $\mu_1(a_2, a_3, \dots,
a_n)$ and $\mu_1(a'_2, a'_3, \dots, a'_n)$ we derive $\mu_1(a_2,
a_3, \dots, a_n)\notin \overline{\Orb(\mu_1(a'_2, a'_3, \dots,
a'_n))}$.

Summarizing these results on single-generated Leibniz algebras, we
obtain
\begin{thm} $\overline{\bigcup\limits_{a_2, \dots, a_n}\Orb(\mu_1(a_2, a_3,
\dots, a_n))}$ is an irreducible component. \end{thm}

\section{Cohomology of Leibniz superalgebras}

In this section we describe all infinitesimal deformations of the
Leibniz superalgebra $NF^{n,m}$ and we prove similar results as in
previous section.

In the next proposition the description of even derivations of
$NF^{n,m}$ is given.

\begin{prop}\label{pr42} Any derivation of $\Der(NF^{n,m})_0$ has the
form:
\begin{align*}
d(y_j) & = (2j-1)a_1y_j+\sum\limits_{k=2}^{m+1-j}a_ky_{j+k-1}, & 1 \leq j \leq  m, \\
d(x_i) & = 2ia_1x_i+\sum\limits_{k=2}^{n+1-i}a_ix_{i+k-1}, & 1 \leq
i \leq n,    \
\end{align*}
 where $m=n$ or $m=n+1$.
\end{prop}

\begin{proof} For $d \in  \Der(NF^{n,m})_0$ we put $d(y_1)=\sum\limits_{k=1}^{m}a_ky_{k}$.
Then using the properties of derivation and multiplication in the
superalgebra $NF^{n,m}$ we obtain $d(x_1) =
2a_1x_1+\sum\limits_{k=2}^{n}a_kx_{k}$.

Using induction, we deduce
$$d(y_{j+1}) = [d(y_j), x_1] + [y_j, d(x_1)] =(2j+1)a_1y_{j+1}+\sum\limits_{k=2}^{m-j}a_ky_{j+k},$$
$$d(x_i) = [d(y_i), y_1] + [y_i, d(y_1)] = 2ia_1x_i+\sum\limits_{k=2}^{n+1-i}a_kx_{i+k-1}.$$

The verification of the derivation property on other elements do not
give any additional restriction on $d$.
\end{proof}

Similarly we describe odd derivations of $\Der(NF^{n,m})$.
\begin{prop}\label{pr43} Any derivation of $\Der(NF^{n,m})_1$ has the
form
$$\begin{array}{ll}d(y_j) = \sum\limits_{k=1}^{n+1-j}b_kx_{j+k-1}, & 1 \leq j \leq  n, \\
d(x_i) = \frac 1 2(b_1y_{i+1}- \sum\limits_{k=2}^{m-i}b_kx_{i+k}),
& 1 \leq i \leq m-1,\end{array}$$ where $m=n$ or $m=n+1$.
\end{prop}

Now we shall consider infinitesimal deformations of the
superalgebra $NF^{n,m}$, i.e., elements of the space
$ZL_0^2(NF^{n,m}, NF^{n,m})$.

\subsection{The case $m=n$}

\

In this case we give the description of the  infinitesimal deformations
of the superalgebra $NF^{n,n}$.

\begin{prop} \label{4.4} An arbitrary infinitesimal deformation $\varphi$ of $NF^{n,n}$
has the following form:
\[\begin{cases} \varphi(y_j, y_1) =
\sum\limits_{k=1}^{n}\alpha_{j,k}x_k, & 1 \leq  j \leq   n,\\
 \varphi(x_j, y_1) = \sum\limits_{k=1}^{n}\beta_{j,k}y_k, & 1 \leq j \leq  n-1,\\
 \varphi(x_n, y_1) = \sum\limits_{k=2}^{n}\beta_{n,k}y_k, &\\
 \varphi(x_j, x_1) = -\alpha_{1,1}x_{i+1} +
 \sum\limits_{k=1}^{n}(\alpha_{j+1,k}+2\beta_{j,k})x_k, &  1 \leq  j \leq  n-1,\\
 \varphi(x_n, x_1) = 2\sum\limits_{k=2}^{n}\beta_{n,k}x_k, & \\
\varphi(y_j, x_1) = 2\beta_{j,1}y_1 - \alpha_{1,1}y_{j+1}+
\sum\limits_{k=2}^{n}(\alpha_{j,k-1}+2\beta_{j,k})y_k, &  1 \leq  j \leq  n-1,\\
\varphi(y_n, x_1) =
\sum\limits_{k=2}^{n}(\alpha_{n,k-1}+2\beta_{n,k})y_k, & \\
 \varphi(x_i, x_{j+1}) = - (\alpha_{j+1,1}+2\beta_{j,1})x_{i+1}, & 1 \leq i \leq  n-1, \ 1 \leq j \leq n-1,\\
\varphi(y_i, x_{j+1}) = - (\alpha_{j+1,1}+2\beta_{j,1})y_{i+1}, & 1 \leq i \leq  n-1, \ 1 \leq j \leq n-1,\\
\varphi(x_i, y_{j+1}) = - \beta_{j,1}y_{i+1}, & 1 \leq i \leq  n-1, \ 1 \leq j \leq n-1,\\
\varphi(y_i, y_{j+1}) = - 2\beta_{j,1}x_{i}, & 1 \leq i \leq  n, \ 1 \leq j \leq n-1.
\end{cases}\]
\end{prop}

\begin{proof} Let $ \varphi \in ZL_0^2(NF^{n,n}, NF^{n,n})$. We set
$$\varphi(y_j, y_1) = \sum\limits_{k=1}^{n} \alpha_{j,k}x_k,\quad
\varphi(x_j, y_1) = \sum\limits_{k=1}^{n} \beta_{j,k}y_k, \quad  1 \leq j \leq n.$$

Applying the multiplication of the superalgebra and the property
of cocycle for $d^2\varphi(x_j, y_1, y_1)=0$, we obtain
$$\varphi(x_j, x_1) = -\alpha_{1,1}x_{j+1} +
 \sum\limits_{k=2}^{n}(\alpha_{j+1,k-1}+2\beta_{j,k})x_k, \ 1 \leq j \leq n-1,\quad \varphi(x_n, x_1) =  2\sum\limits_{k=1}^{n}\beta_{n,k}x_k.$$

Analogously, from $d^2\varphi(y_j, y_1, y_1)=0$ we get
$$\varphi(y_j, x_1) = 2\beta_{j,1}y_1 - \alpha_{1,1}y_{j+1}+
\sum\limits_{k=2}^n(\alpha_{j,k-1} + 2\beta_{j,k})y_k, \quad 1 \leq j \leq n.$$

The equations $d^2\varphi(x_i, x_1, x_1)=0$ and $d^2\varphi(y_i,
x_1, x_1)=0$ imply
\begin{align*}
\varphi(x_i, x_2) &= - [x_i,\varphi(x_1, x_1)] = -(\alpha_{2,1}+2\beta_{1,1})x_{i+1}, && 1 \leq i \leq n-1,\\
\varphi(y_i, x_2) &= - [y_i,\varphi(x_1, x_1)] = -(\alpha_{2,1}+2\beta_{1,1})y_{i+1}, &&  1 \leq i \leq n-1.
\end{align*}

Using the conditions $d^2 \varphi(x_i, x_1, x_j) = 0$ and $d^2
\varphi(x_i, x_j, x_1) = 0$ for $1 \leq i \leq n$, $ 2 \leq j \leq
n$, we derive
\begin{align*}
[x_i,\varphi(x_1,x_j)] + [\varphi(x_i,x_j), x_1] -  \varphi([x_i,x_1],x_j) &=0,\\
[x_i,\varphi(x_j,x_1)] - [ \varphi(x_i,x_j), x_1]  +
\varphi(x_i, [x_j,x_1])  +  \varphi([x_i,x_1],x_j)&=0.
\end{align*}

Summarizing these equalities, we deduce
$$\varphi(x_i, x_{j+1}) = - [x_i,  \varphi(x_1, x_j)+\varphi(x_j, x_1)] = -(\alpha_{j+1,1}+2\beta_{j,1})x_{i+1}, \quad 1 \leq i \leq n-1, \ 2 \leq j \leq n-1,$$
and $0=[x_i,  \varphi(x_1, x_n)+\varphi(x_n, x_1)] = \beta_{n,1}y_{i+1}$, which implies $\beta_{n,1}=0$.

Similarly from $d^2 \varphi(y_i, x_1, x_j) = 0$ and $d^2
\varphi(y_i, x_j, x_1) = 0$ we obtain
$$\varphi(y_i, x_{j+1}) = - (\alpha_{j+1,1}+2\beta_{j,1})y_{i+1}, \quad 1 \leq i \leq n-1, \ 2 \leq j \leq n-1.$$

Considering the properties $(d^2 \varphi)(x_i, y_1, x_j) = 0$ and
$(d^2 \varphi)(x_i, x_j, y_1) = 0$ for $1 \leq i,j \leq n$, we
have
$$[x_i,\varphi(y_1,x_j)] - [\varphi(x_i,y_1), x_j] + [\varphi(x_i,x_j), y_1] + \varphi(x_i, [y_1,x_j])-  \varphi([x_i,y_1],x_j) + \varphi([x_i,x_j],y_1)=0,$$
$$[x_i,\varphi(x_j, y_1)] - [\varphi(x_i,x_j), y_1] + [\varphi(x_i,y_1), x_j] + \varphi(x_i, [x_j,y_1])-  \varphi([x_i,x_j],y_1) + \varphi([x_i,y_1],x_j)=0.$$

Again, summarizing these equalities, we get $\varphi(x_i, [y_1,
x_j]+[x_j, y_1]) = - [x_i, \varphi(y_1, x_j)+\varphi(x_j, y_1)]$,
from which we have
\begin{align*}
\varphi(x_i, y_{2}) & = - \frac 2 3 [x_i,
\varphi(y_1, x_1)+\varphi(x_1, y_1)] = -\beta_{1,1}y_{i+1}, &&
1 \leq i \leq n-1,\\
\varphi(x_i, y_{j+1}) & = - 2[x_i,  \varphi(y_1, x_j)+\varphi(x_j, y_1)] = -\beta_{1,1}y_{i+1}, && 1 \leq i \leq n-1, \ 2 \leq j \leq n-1.
\end{align*}

Applying the above arguments to the equalities $(d^2 \varphi)(y_i,
y_1, x_j) = 0$ and $(d^2 \varphi)(y_i, x_j, y_1) = 0$ for $1 \leq
i,j \leq n$, we get
$$\varphi(y_i, y_{j+1}) =  -2\beta_{j,1}x_{i}, \quad 1 \leq i \leq n, \ 1 \leq j \leq n-1.$$

Checking the general condition of cocycle for the other basis
elements we get already obtained restrictions.
\end{proof}

Using the assertion of Proposition \ref{4.4} we indicate a basis
of the space $ZL_0^2(NF^{n,n}, NF^{n,n})$.

\begin{thm}  The following cochains $\varphi_{j,k,}, \psi_{j,k}:$
$$\varphi_{1,1}:\left\{\begin{array}{ll}\varphi_{1,1}(y_1, y_1) = x_1, & \\
 \varphi_{1,1}(x_i, x_1) = -x_{i+1}, & 1 \leq i \leq n-1,\\
 \varphi_{1,1}(y_i, x_1) = -y_{i+1}, & 2 \leq i \leq
 n-1,\end{array}\right. \
 \varphi_{j,1}(2 \leq j \leq n): \left\{\begin{array}{ll} \varphi_{j,1}(y_j, y_1) = x_1,&\\
\varphi_{j,1}(x_{j-1}, x_1) = x_1, & \\ \varphi_{j,1}(y_j, x_1) = y_2, & \\
\varphi_{j,1}(x_i, x_j) = -x_{i+1},& 1 \leq i \leq n-1,\\
\varphi_{j,1}(y_i, x_j) = -y_{i+1}, & 1 \leq i \leq
n-1,\end{array}\right.$$

$$\begin{array}{ll}
\varphi_{1,k}(2 \leq k \leq n-1):\left\{\begin{array}{l}\varphi_{1,k}(y_1, y_1) = x_k, \\
 \varphi_{1,k}(y_1, x_1) = y_{k+1},\end{array}\right. &
 \varphi_{1,n}:\left\{\begin{array}{l}\varphi_{1,n}(y_1, y_1) =
 x_n.\end{array}\right. \end{array}$$

$$\begin{array}{ll} \varphi_{j,k}\left(\begin{array}{l}2 \leq j \leq n, \\ 2 \leq k \leq n-1\end{array}\right): \left\{\begin{array}{l} \varphi_{j,k}(y_j, y_1) = x_k,\\
\varphi_{j,k}(x_{j-1}, x_1) = x_k, \\
 \varphi_{j,k}(y_j, x_1) = y_{k+1},\end{array}\right. &
 \varphi_{j,n}(2 \leq j \leq n): \left\{\begin{array}{l} \varphi_{j,n}(y_j, y_1) = x_n,\\
\varphi_{j,n}(x_{j-1}, x_1) = x_n,\end{array}\right.\end{array}$$

\[ \psi_{j,1}(1 \leq j \leq n-1):\left\{\begin{array}{ll}\psi_{j,1}(x_j, y_1) = y_1, & \\
\psi_{j,1}(x_j, x_1) = 2x_{1}, & \\
\psi_{j,1}(y_j, x_1) = 2y_{1}, & \\
\psi_{j,1}(x_i, x_{j+1}) = -2x_{i+1}, & 1 \leq i \leq n-1,
\\ \psi_{j,1}(y_i, x_{j+1}) = -2y_{i+1}, & 1 \leq i \leq n-1,\\
\psi_{j,1}(x_i, y_{j+1}) = -y_{i+1}, & 1 \leq i \leq n-1,\\
\psi_{j,1}(y_i, y_{j+1}) = -2x_{i}, & 1 \leq i \leq n,
\end{array}\right.
\psi_{j,k}\left(\begin{array}{l}1 \leq j \leq n, \\ 2 \leq k \leq n\end{array}\right):\left\{\begin{array}{l}\psi_{j,k}(x_j, y_1) = y_k,\\
\psi_{j,k}(x_j, x_1) = 2x_{k},\\
\psi_{j,k}(y_j, x_1) = 2y_{k}.\end{array}\right.
\]

form a basis of the space $ZL_0^2(NF^{n,n}, NF^{n,n})$.
\end{thm}

Applying the same arguments as used in the proof of Proposition~\ref{pr4}
we prove the following result.

\begin{prop} The 2-cochains $\xi_{j,k}$ and
$\zeta_{j,k}$ defined as follows:
$$\left\{\begin{array}{lll}\xi_{j,k} =  \varphi_{j,k}, & 1 \leq j \leq n, & j \leq k \leq n,\\
\xi_{j,k} =  \varphi_{j,k} - \frac 1 2 \psi_{j,k+1},& 2 \leq j \leq n, & 1 \leq k \leq j-1,\\
\zeta_{j,k} =  \psi_{j-1,k},& 2 \leq j \leq n, & j \leq k \leq n,\\
\zeta_{j,k} =  \frac 1 2 \psi_{j-1,k} - \varphi_{j,k},& 2 \leq j \leq n, & 1 \leq k \leq j-1,\\
\end{array}\right.$$
form a basis of $BL_0^2(NF^{n,n}, NF^{n,n})$.
\end{prop}

\begin{cor} $\{\overline{\psi_{n,2}},\overline{\psi_{n,3}}  , \dots,
\overline{\psi_{n,n}}\}$ form a basis of $HL_0^2(NF^{n}, NF^{n})$.
\end{cor}

Consequently, $$\begin{array}{l}\dim ZL_0^2(NF^{n,n}, NF^{n,n}) =
2n^2 - 1, \\ \dim BL_0^2(NF^{n,n}, NF^{n,n})=2n^2-n, \\ \dim
HL^2_0(NF^{n}, NF^{n}) = n - 1.\end{array}$$

In the next proposition we clarify taht basis element of
$ZL_0^2(NF^{n,n},NF^{n,n})$ satisfies the condition \eqref{E:2.1}.

\begin{prop} The infinitesimal deformations $\varphi_{j,k} \ (1 \leq j \leq n, 2 \leq k \leq n)$
and $\psi_{j,k} \ (1 \leq j \leq n, 2 \leq k \leq n)$ satisfy the
condition \eqref{E:2.1}, but the 2-cocycles $\varphi_{j,1} \ (1 \leq j \leq
n)$ and $\psi_{j,1} \ (1 \leq j \leq n-1)$ do not satisfy the
condition \eqref{E:2.1}.
\end{prop}

\begin{proof} The proof of this proposition is carry out by applying similar arguments as in the
proof of Proposition \ref{pr35}.
\end{proof}

\subsection{The case $m=n+1$}

\

In this subsection we investigate the case $m=n+1$. Below we omit the
proofs of results of this subsection, because of they are obtained
similarly to above.

\begin{prop}\label{pr4.9} Any 2-cocycle $\varphi \in ZL_0^2(NF^{n,n+1}, NF^{n,n+1})$
has the following form:
\[\begin{cases} \varphi(y_j, y_1) =
\sum\limits_{k=1}^{n}\alpha_{j,k}x_k, & 1 \leq  j \leq   n+1,\\
 \varphi(x_j, y_1) = \sum\limits_{k=1}^{n+1}\beta_{j,k}y_k, & 1 \leq j \leq  n-1,\\
 \varphi(x_n, y_1) = - \frac{\alpha_{n+1,1}}{2}y_1 + \sum\limits_{k=2}^{n}\beta_{n,k}y_k, &\\
 \varphi(x_j, x_1) = -\alpha_{1,1}x_{i+1} +
 \sum\limits_{k=1}^{n}(\alpha_{j+1,k}+2\beta_{j,k})x_k, &  1 \leq  j \leq  n-1,\\
 \varphi(x_n, x_1) = \sum\limits_{k=2}^{n}(\alpha_{n+1,k}+2\beta_{n,k})x_k, & \\
\varphi(y_j, x_1) = 2\beta_{j,1}y_1 - \alpha_{1,1}y_{j+1}+
\sum\limits_{k=2}^{n+1}(\alpha_{j,k-1}+2\beta_{j,k})y_k, &  1 \leq  j \leq  n-1,\\
\varphi(y_n, x_1) = \alpha_{n+1,1}y_1 - \alpha_{1,1}y_{n+1}+
\sum\limits_{k=2}^{n+1}(\alpha_{n,k-1}+2\beta_{n,k})y_k, & \\
\varphi(y_{n+1}, x_1) = \sum\limits_{k=2}^{n+1}\alpha_{n+1,k-1}y_k, & \\
 \varphi(x_i, x_{j+1}) = - (\alpha_{j+1,1}+2\beta_{j,1})x_{i+1}, & 1 \leq i \leq  n-1, \ 1 \leq j \leq n-1,\\
\varphi(y_i, x_{j+1}) = - (\alpha_{j+1,1}+2\beta_{j,1})y_{i+1}, & 1 \leq i \leq  n, \ 1 \leq j \leq n-1,\\
\varphi(x_i, y_{j+1}) = - \beta_{j,1}y_{i+1}, & 1 \leq i \leq  n, \ 1 \leq j \leq n-1,\\
\varphi(x_i, y_{n+1}) = - \frac {\alpha_{n+1,1}} 2 y_{i+1}, & 1 \leq i \leq  n,\\
\varphi(y_i, y_{j+1}) = - 2\beta_{j,1}x_{i}, & 1 \leq i \leq  n, \
1 \leq j \leq n-1,\\
\varphi(y_i, y_{n+1}) = - \alpha_{n+1,1}x_{i}, & 1 \leq i \leq  n.\\
\end{cases}\]
\end{prop}

Using the assertion of Proposition \ref{pr4.9} we indicate a
basis of the space $ZL_0^2(NF^{n,n+1}, NF^{n,n+1})$.
\begin{thm} The following cochains $\varphi_{j,k,}, \psi_{j,k}:$
$$\varphi_{1,1}:\left\{\begin{array}{ll}\varphi_{1,1}(y_1, y_1) = x_1, & \\
 \varphi_{1,1}(x_i, x_1) = -x_{i+1}, & 1 \leq i \leq n-1,\\
 \varphi_{1,1}(y_i, x_1) = -y_{i+1}, & 2 \leq i \leq n,\end{array}\right.
  \
 \varphi_{j,1}(2 \leq j \leq n): \left\{\begin{array}{ll} \varphi_{j,1}(y_j, y_1) = x_1,&\\
\varphi_{j,1}(x_{j-1}, x_1) = x_1, & \\
 \varphi_{j,1}(y_j, x_1) = y_2, & \\
\varphi_{j,1}(x_i, x_j) = -x_{i+1},& 1 \leq i \leq n-1,\\ \varphi_{1,1}(y_i, x_j) = -y_{i+1}, & 1 \leq i \leq n,\end{array}\right.$$
$$\varphi_{n+1,1}: \left\{\begin{array}{ll} \varphi_{n+1,1}(y_{n+1}, y_1) = x_1,&\\
\varphi_{n+1,1}(x_{n}, y_1) = - \frac 1 2 y_1, & \\
 \varphi_{n+1,1}(y_n, x_1) = -y_1, & \\
\varphi_{n+1,1}(y_{n+1}, x_1) = y_2,& \\
\varphi_{n+1,1}(x_i, y_{n+1}) = - \frac 1 2 y_{i+1},& 1 \leq i \leq n, \\
\varphi_{n+1,1}(y_i, y_{n+1}) = - x_{i}, & 1 \leq i \leq n,\\
\end{array}\right.$$
$$\varphi_{1,k}(2 \leq k \leq n):\left\{\begin{array}{l}\varphi_{1,k}(y_1, y_1) = x_k, \\
 \varphi_{1,k}(y_1, x_1) = y_{k+1},\end{array}\right.  \
 \varphi_{j,k}\left(\begin{array}{l}2 \leq j \leq n+1,  \\ 2 \leq k \leq n\end{array}\right): \left\{\begin{array}{l} \varphi_{j,k}(y_j, y_1) = x_k,\\
\varphi_{j,k}(x_{j-1}, x_1) = x_k, \\
 \varphi_{j,k}(y_j, x_1) = y_{k+1},\end{array}\right.$$
$$\psi_{j,1}(1 \leq j \leq n-1):\left\{\begin{array}{ll}\psi_{j,1}(x_j, y_1) = y_1, & \\
\psi_{j,1}(x_j, x_1) = 2x_{1}, & \\
\psi_{j,1}(y_j, x_1) = 2y_{1}, & \\
\psi_{j,1}(x_i, x_{j+1}) = -2x_{i+1}, & 1 \leq i \leq n-1,
\\
\psi_{j,1}(y_i, x_{j+1}) = -2y_{i+1}, & 1 \leq i \leq n,\\
\psi_{j,1}(x_i, y_{j+1}) = -y_{i+1}, & 1 \leq i \leq n,\\
\psi_{j,1}(y_i, y_{j+1}) = -2x_{i}, & 1 \leq i \leq n,
\end{array}\right.
$$
$$\psi_{j,k}\left(\begin{array}{l} 1 \leq j \leq n, \\ 2 \leq k \leq n\end{array}\right):\left\{\begin{array}{l}\psi_{j,k}(x_j, y_1) = y_k,\\
\psi_{j,k}(x_j, x_1) = 2x_{k},\\
\psi_{j,k}(y_j, x_1) = 2y_{k},\end{array}\right.
\psi_{j,n+1}(1 \leq j \leq n):\left\{\begin{array}{l}\psi_{j,n+1}(x_j, y_1) = y_{n+1},\\
\psi_{j,n+1}(y_j, x_1) = 2y_{n+1},\end{array}\right.
$$
form a basis of $ZL_0^2(NF^{n,n+1}, NF^{n,n+1})$.
\end{thm}

\begin{prop} The cochains $\xi_{j,k}$ and
$\zeta_{j,k}$ defined as:
$$\left\{\begin{array}{lll}\xi_{j,k} =  \varphi_{j,k}, & 1 \leq j \leq n, & j \leq k \leq n,\\
\xi_{j,k} =  \varphi_{j,k} - \frac 1 2 \psi_{j,k+1},& 2 \leq j \leq n, & 1 \leq k \leq j-1,\\
\zeta_{j,1} =  \frac 1 2 \psi_{j-1,1} - \varphi_{j,1},& 2 \leq j \leq n,\\
\zeta_{n+1,1} =   - \varphi_{n+1,1},&\\
\zeta_{j,k} =  \psi_{j-1,k},& 2 \leq j \leq n+1, & j \leq k \leq n+1,\\
\zeta_{j,k} =  \frac 1 2 \psi_{j-1,k} - \varphi_{j,k},& 2 \leq j \leq n+1, & 2 \leq k \leq j-1,\\
\end{array}\right.$$
form a basis of $BL_0^2(NF^{n,n+1}, NF^{n,n+1})$.
\end{prop}

\begin{cor} $\{\overline{\varphi_{n+1,2}},\overline{\varphi_{n+1,3}} , \dots,
\overline{\varphi_{n+1,n}}\}$ form a basis of $HL_0^2(NF^{n,n+1},
NF^{n,n+1})$.
\end{cor}

 Therefore
 $$\begin{array}{l}\dim ZL_0^2(NF^{n,n+1},
NF^{n,n+1}) = 2n^2 + 2n-1, \\ \dim BL_0^2(NF^{n,n+1}, NF^{n,n+1}) =
2n^2 + n, \\ \dim HL^2_0(NF^{n,n+1}, NF^{n,n+1}) = n - 1.
\end{array}$$

Below, we indicate a basis infinitesimal deformations satisfying
the condition \eqref{E:2.1}.

\begin{prop} The infinitesimal deformations $\varphi_{j,k} \ (1 \leq j \leq n+1, 2 \leq k \leq n)$
and $\psi_{j,k} \ (1 \leq j \leq n, 2 \leq k \leq n+1)$ satisfy
the condition \eqref{E:2.1}, but the 2-cocycles $\varphi_{j,1} \ (1 \leq j
\leq n+1)$ and $\psi_{j,1} \ (1 \leq j \leq n-1)$ do not satisfy
the condition \eqref{E:2.1}.
\end{prop}

Since $\sum\limits_{k=2}^nb_k\psi_{n,k}$ and
$\sum\limits_{k=2}^nc_k\varphi_{n+1,k}$ define  linear integrable
deformations of $NF^{n,n}$ and $NF^{n,n+1}$, respectively, we
consider two families of superalgebras $\nu_t(b_2, b_3, \dots,
b_n) = NF^{n,n} + t\sum\limits_{k=2}^nb_k\psi_{n,k}$ and
$\eta_t(c_2, c_3, \dots, c_n) = NF^{n,n+1} +
t\sum\limits_{k=2}^nc_k\varphi_{n+1,k}$ with the
multiplication tables
\[\begin{cases}[y_i, y_1] = x_{i},& 1 \leq i \leq n,\\
[x_i, y_1] = \frac 1 2 y_{i+1},& 1 \leq i \leq n-1,\\
[x_n, y_1] = t\sum\limits_{k=2}^nb_ky_k,&\\
[y_i, x_1] = y_{i+1},& 1 \leq i \leq n-1,\\
[y_n, x_1] = 2t\sum\limits_{k=2}^nb_ky_k,& \\
[x_i, x_1] = x_{i+1},& 1 \leq i \leq n-1,\\
[x_n, x_1] = 2t\sum\limits_{k=2}^nb_kx_{k},\\
\end{cases} \ \mbox{and} \  \begin{cases}[x_i, x_1] = x_{i+1},& 1 \leq i \leq n-1,\\
[x_{n}, x_1] = t\sum\limits_{k=j}^nc_kx_k,&\\
[y_i, x_1] = y_{i+1},& 1 \leq j \leq n,\\
[y_{n+1}, x_1] = t\sum\limits_{k=j}^nc_ky_{k+1},&\\
[y_i, y_1] = x_{i},& 1 \leq i \leq n,\\
[y_{n+1}, y_1] = t\sum\limits_{k=j}^nc_kx_k,&\\
[x_i, y_1] = \frac 1 2 y_{i+1},& 1 \leq i \leq n.
\end{cases}\] respectively.

Putting $b_k' = tb_k$ and $c_k'=tc_k$, we can assume in both
multiplications $t= 1$.

The description of single-generated Leibniz superalgebras deduce
that they are have the forms of superalgebras $\nu_1(b_2, b_3,
\dots, b_n)$ and $\eta_1(c_2, c_3, \dots, c_n)$.

Similarly to the case of Leibniz algebras for superalgebras we
obtain the following theorem.
\begin{thm} $\overline{\bigcup\limits_{b_2, \dots, b_n}\Orb(\nu_1(b_2, b_3,
\dots, b_n))}$ and  $\overline{\bigcup\limits_{c_2, \dots,
c_n}\Orb(\eta_1(c_2, c_3, \dots, c_n))}$ are irreducible
components of the varieties $Leib^{n,n}$ and $Leib^{n,n+1}$,
respectively.
\end{thm}


\section*{Acknowledgements}
The authors are grateful to Prof. A.S. Dzhumadildaev for providing
the derived map $d$ for right Leibniz superalgebras. Authors are also
thankful to Prof. V.V. Gorbatsevich for his useful comments. The
second author was partially supported by Grant (RGA) No:11-018
RG/Math/AS\_I--UNESCO FR: 3240262715.


\begin{thebibliography}{9}
\bibitem{Alb} Albeverio S., Ayupov Sh.A., Omirov B.A.
{\it On nilpotent and simple Leibniz algebras}, Comm.  Algebra
33(1) (2005), p. 159--172.


\bibitem{Ayup} Ayupov Sh.A., Omirov B.A. {\it On some classes of nilpotent
Leibniz algebras}, Sib. Math. J. 42(1) (2001), p. 18--29.

\bibitem{Balavoine} Balavoine D. {\it D\'eformations et rigidit\'e
g\'eom\'etrique des algebras de Leibniz}, Comm.  Algebra  24 (1996), p. 1017--1034.

\bibitem{Burde1} Burde D. {\it Degeneration of $7$-dimensional nilpotent Lie algebras}, Comm.
Algebra 33(4) (2005), p. 1259--1277.


\bibitem{Cas1} Casas J. M., Ladra M., Omirov B.A., Karimjanov I.K.
{\it Classification of solvable Leibniz algebras with
null-filiform nilradical}, Linear Algebra Appl.  438 (7) (2013), p. 2973--3000.

\bibitem{Dzhum} Dzhumadil'daev A.S. {\it Cohomologies of colour Leibniz algebras: Pre-simplicial approach}, Lie
Theory and its applications in physics III, Proceeding of the
Third International Workshop (1999), p. 124--135

\bibitem{Fial1} Fialowski A. {\it Deformations in Mathematics and Physics}, Intern.
Journal of Theor. Physics, 47(2) (2008), p. 333--337.

\bibitem{Fial2} Fialowski A., Millionschikov D.V. {\it Cohomology of Graded Lie
Algebras of Maximal Class}, J. Algebra, 296 (2006), p.
157--176.

\bibitem{Fial3} Fialowski, A., Penkava, M.
{\it Formal deformations, contractions and moduli spaces of Lie
algebras}, Internat. J. Theoret. Phys. 47
(2008), p. 561--582.

\bibitem{Gersten} Gerstenhaber M. {\it On the deformation of rings and
algebras, I, III}, Ann. of Math. (2) 79 (1964), p. 59--103; 88 (1968), p. 1--34.

\bibitem{Khak1} G\'{o}mez J.R., Khakimdjanov Yu., Navarro R.M. {\it Some problems concerning to nilpotent Lie superalgebras}, J. Geom. Phys. 51(4) (2004), p.
    473--486.

\bibitem{Khak} Khakimdjanov Yu., Navarro R.M. {\it Deformations of filiform Lie
algebras and superalgebras}, J. Geom. Phys. 60 (2010), p.
1156--1169.

\bibitem{Lo 2} Loday J.-L. {\it Une version non commutative des alg\`ebres de Lie: les
alg\`ebres de Leibniz}, Enseign. Math. (2)  39  (1993),  p. 269--293.

\bibitem{Lo 3} Loday J.-L.,  Pirashvili T. {\it Universal enveloping algebras
of Leibniz algebras and (co)homology}, Math. Ann. 296 (1993),
p. 139--158.

\bibitem{Mil} Millionschikov D. V. {\it Deformations of filiform Lie algebras and
symplectic structures}, Proc. Steklov Inst. Math. 2006, vol. 252, Issue 1, p. 182--204.

\bibitem{Nijen} Nijenhuis A., Richardson R. W. {\it Cohomology and
deformations in graded Lie algebras}, Bull. Amer. Math. Soc.
72 (1966), p. 1--29.
\end{thebibliography}
\end{document}